\documentclass[pdftex,12pt,a4paper]{article}

\pdfoutput=1 

\usepackage{authblk}
\usepackage[utf8]{inputenc}
\usepackage{amsmath}
\usepackage{amssymb}
\usepackage{amsthm}
\usepackage{amsrefs}
\usepackage{hyperref}
\usepackage{tikz}

\newtheorem{theorem}{Theorem}[section]
\newtheorem{lemma}[theorem]{Lemma}
\newtheorem{proposition}[theorem]{Proposition}
\newtheorem{corollary}[theorem]{Corollary}
\newtheorem{remark}[theorem]{Remark}
\newtheorem{definition}[theorem]{Definition}
\newtheorem{algorithm}[theorem]{Algorithm}

\newcommand{\Mod}[1]{\ (\mathrm{mod}\ #1)}

\title{Two-point AG codes from the Beelen-Montanucci maximal curve}
\author{Leonardo Landi}
\author{Lara Vicino}
\affil{Department of Applied Mathematics and Computer Science, Technical University of Denmark, 2800 Kgs. Lyngby, Denmark lelan@dtu.dk - ORCID 0000-0003-4835-5321\\ lavi@dtu.dk - ORCID 0000-0002-7480-8681}
\date{}

\begin{document}
\maketitle

\begin{abstract}
  In this paper we investigate two-point algebraic-geometry codes (AG codes) coming from the Beelen-Montanucci (BM) maximal curve. We study properties of certain two-point Weierstrass semigroups of the curve and use them for determining a lower bound on the minimum distance of such codes. AG codes with better parameters with respect to comparable two-point codes from the Garcia-G\"{u}neri-Stichtenoth (GGS) curve are discovered.
\end{abstract}



\section{Introduction}
\label{section:introduction}

Algebraic geometry codes, or simply \textit{AG codes}, are a family of error-correcting codes introduced by Goppa in the '80s (see \cite{Goppa}, \cite{GoppaBound}) and constructed using algebraic curves defined over a finite field. AG codes provide good examples of error-correcting codes when compared to other linear codes such as Reed-Solomon codes (RS codes). The basic parameters of an AG code $\mathcal{C}$ are its length $n_{\mathcal{C}}$, its dimension $k_{\mathcal{C}}$ and its minimum distance $d_{\mathcal{C}}$. The minimum distance $d_{\mathcal{C}}$ describes the error-correcting capability of the code, thus it is desirable for applications to construct codes with large minimum distance. A general lower bound for the minimum distance of an AG code is given by the well-known Goppa bound; a consequence of this bound is that for a code whose underlying algebraic curve has genus $g$, the inequality $d_{\mathcal{C}} \geq n_{\mathcal{C}} - k_{\mathcal{C}} + 1 - g$ holds, hence the minimum distance can be designed.

Let $\mathbb{F}_q$ be the finite field with $q$ elements and $\mathcal{X}$ be an algebraic curve defined over $\mathbb{F}_q$ and of genus $g$. $\mathcal{X}$ is said to be \textit{maximal} if it attains the Hasse-Weil bound $|\mathcal{X}(\mathbb{F}_q)| \leq q + 1 + 2g \sqrt{q}$, where $|\mathcal{X}(\mathbb{F}_q)|$ is the number of $\mathbb{F}_q$-rational points of $\mathcal{X}$. In other words, a maximal curve has the largest number of rational points with respect to its genus. For this reason maximal curves are suitable candidates for the construction of AG codes with good parameters. 

An important example of maximal curve over a finite field is the Hermitian curve
\begin{equation}
  \label{eq:Hermitian}
  \mathcal{H}_q : y^{q+1} = x^{q+1} - 1,
\end{equation}
often defined by the equivalent affine equation $y^{q+1} = x^q + x$, which is maximal over $\mathbb{F}_{q^2}$. The Hermitian curve has been intensively studied and, together with the Suzuki and the Ree curves, it forms a family of maximal curves over suitable finite fields called the Deligne-Lusztig curves.

Another important example of maximal curve is the $GK$ curve, constructed by Giulietti and Korchm\'{a}ros in \cite{GKCurve}. The $GK$ curve is defined over $\mathbb{F}_{q^6}$ by the affine equations
$$ GK :
\begin{cases}
  y^{q+1} = x^q + x, \\
  z^{q^2-q+1} = y \frac{x^{q^2}-x}{x^q+x}.
\end{cases} $$
A generalization of the $GK$ curve to an infinite family of $\mathbb{F}_{q^{2n}}$-maximal curves $GGS_n$ for $n \geq 3$ odd has been given by Garcia, G\"{u}neri and Stichtenoth in \cite{GGSCurve}. The $GGS_n$ curves are defined by affine equations
$$ GGS_n :
\begin{cases}
  y^{q+1} = x^q + x, \\
  z^m = y \frac{x^{q^2}-x}{x^q+x},
\end{cases} $$
where $m := (q^n+1)/(q+1)$. The $GK$ curve is a special case of the $GGS_n$ curve for $n=3$.

A different generalization of the $GK$ curve was introduced by Beelen and Montanucci in \cite{ANewFamily}. For a prime power $q$ and an odd integer $n \geq 3$, the Beelen-Montanucci curve $\mathcal{BM}_n$ is defined over $\mathbb{F}_{q^{2n}}$ by the affine equations
\begin{equation}
  \label{eq:BM_definition}
  \mathcal{BM}_n :
  \begin{cases}
    y^{q+1} = x^{q+1} - 1, \\
    z^m = y \frac{x^{q^2}-x}{x^{q+1}-1},
  \end{cases}
\end{equation}
where $m := (q^n+1)/(q+1)$. The curve $\mathcal{BM}_n$ is maximal over the field $\mathbb{F}_{q^{2n}}$ and for $n=3$ it is isomorphic to the $GGS_n$ curve and, equivalently, to the $GK$ curve. Further, $\mathcal{BM}_n$ is isomorphic to $GGS_n$ if and only if $n=3$, as shown in \cite{ANewFamily}.

For dual codes of two-point AG codes, methods that give a lower bound on the minimum distance that is possibly better than the Goppa bound were studied by Matthews in \cite{Matthews} and Beelen in \cite{OrderBound}; both involve a generalization of the Weierstrass semigroup at one point to pairs of points. Over the years, these techniques have been applied to specific maximal curves: for example this is the case of the Hermitian curve in \cite{HommaKim}, \cite{Park} and \cite{DuursmaKirov}, the Suzuki curve in \cite{TwoPointsSuzuki}, the $GK$ curve in \cite{TwoPointsGK} and the $GGS_n$ curves in \cite{TwoPointCodes}.

The aim of this paper is to study duals of two-point AG codes coming from the $\mathcal{BM}_n$ curve. Our approach is similar to the one developed in \cite{TwoPointCodes}, in which the $GGS_n$ curves are considered. In particular, the order bound introduced in \cite{OrderBound} will be used to compute a lower bound on the minimum distance that improves the Goppa bound. In Section \ref{section:preliminary_results}, general results on AG codes will be presented, with a particular focus on two-point AG codes. Section \ref{section:two_point_semigroup} is dedicated to the study of a certain two-point Weierstrass semigroup, as defined in \cite{GeneralizationWeierstrassSemigroup}, on the $\mathcal{BM}_n$ curve. The fourth and last Section is devoted to the computation of the order bound for duals of two-point AG codes coming from the $\mathcal{BM}_n$ curve and includes results for specific values of $q$ and $n$.

\section{Preliminary results}
\label{section:preliminary_results}

We recall some notations and results for two-point AG codes. A more general exposition of AG codes can be found in \cite{Stichtenoth}*{Chapter 2}. Throughout this section, let $q$ be any prime power. For an algebraic curve $\mathcal{X}$ defined over the finite field $\mathbb{F}_q$ of genus $g(\mathcal{X})$, denote with $\mathbb{F}_q(\mathcal{X})$ the field of functions on $\mathcal{X}$. For a fixed positive integer $n$, let $P_1, \dots, P_n$ be rational points of $\mathcal{X}$ and let $D$ be the divisor $D := P_1 + \dots + P_n$. Further, let $G$ be another divisor whose support is disjoint from the support of $D$. The Riemann-Roch space associated to $G$ is the $\mathbb{F}_q$-vector space 
$$ L(G) := \{ f \in \mathbb{F}_q(\mathcal{X}) \mid (f) + G \geq 0 \} \cup \{ 0 \}. $$
The AG code $C_L(D, G)$ is defined as 
$$ C_L(D, G) := \{ (f(P_1), \dots, f(P_n)) \mid f \in L(G) \}. $$
$C_L(D, G)$ is a linear subspace of $\mathbb{F}_q^n$ of dimension $\mathrm{dim}(C_L(D, G)) = \mathrm{dim}(L(G)) - \mathrm{dim}(L(G-D))$ and its minimum distance $d$ satisfies the bound $d \geq n - \mathrm{deg}(G)$. Further, we define $C_L(D, G)^\perp$ as the dual of $C_L(D, G)$; $C_L(D, G)^\perp$ is a linear subspace of $\mathbb{F}_q^n$ of dimension $n - \mathrm{dim}(C_L(D, G))$ and its minimum distance $d^\prime$ satisfies the bound $d^\prime \geq \mathrm{deg}(G) - 2g(\mathcal{X}) + 2$.

The lower bound $d^\prime \geq \mathrm{deg}(G) - 2g(\mathcal{X}) + 2$ is generally not tight and can be improved. To this aim, an approach similar to the one used in \cite{OrderBound} can be used. We write $G = F_1 + F_2$ for some divisors $F_1$ and $F_2$ of $\mathcal{X}$ and we define for a rational point $R$ of $\mathcal{X}$
\begin{align*}
  H(R; G) & := \{ -v_R(f) \mid f \in L(G + \infty R) \setminus \{ 0 \}\}, \\
  N(R; F_1, F_2) & := \{ (i, j) \in H(R; F_1) \times H(R; F_2) \mid i+j = v_R(G)+1 \}, \\
  \nu(R; F_1, F_2) & := \# N(R; F_1, F_2),
\end{align*}
where  $v_R(G)$ denotes the coefficient of $R$ in the divisor $G$ and $L(G + \infty R) := \bigcup_{i\in \mathbb{Z}}L(G + iR)$. $H(R; G)$ is called the set of \textit{$G$-non-gaps at $R$}, studied for example in \cite{ConsecutiveGaps}. Observe that $H(R; G + R) = H(R; G)$. Also, observe that $H(R; 0)$ is the Weierstrass semigroup $H(R)$ at $R$. 

A convenient choice for the divisors $F_1$ and $F_2$ that we will keep for the rest of the paper is $F_1 = 0$ and $F_2 = G$. With these definitions in place, we recall now the generalized order bound introduced in \cite{OrderBound}*{Section 2}, to which we refer for a detailed discussion.

\begin{definition}[\cite{OrderBound}*{Definition 6}]
  \label{def:orderbound}
  Let $D$ be a divisor that is a sum of $n$ distinct rational points of $\mathcal{X}$ and $G$ a divisor on $\mathcal{X}$ such that all the points in its support are rational. Further, suppose that the support of $G$ is disjoint from the support of $D$. For any infinite sequence $S = R_1, R_2,\ldots$ of points of $\mathcal{X}\setminus \mathrm{supp}(D)$ define
  $$ d_S(G) := \mathrm{min}\{\nu(R_{i+1}; 0, G + R_1 + \cdots + R_i)\}, $$
  where the minimum is taken over all $i\geq 0$ such that $L(G + R_1 + \cdots + R_i) \neq L(G + R_1 + \cdots + R_{i+1})$. Moreover, define 
  $$ d(G):=\mathrm{max} \ d_S(G), $$
  where the maximum is taken over all infinite sequences $S$ of points having entries in $\mathcal{X}\setminus \mathrm{supp}(D)$.
\end{definition}

From \cite{OrderBound}*{Theorem 7}, the following proposition holds.

\begin{proposition}
  \label{prop:orderbound}
  Let $C_L(D, G)$ be an AG code with $D$ and $G$ as in Definition \ref{def:orderbound}. Then the minimum distance $d^\prime$ of the dual code $C_L(D, G)^\perp$ satisfies the inequality $d^\prime \geq d(G)$.
\end{proposition}

By virtue of Proposition \ref{prop:orderbound} we will refer to the quantity $d(G)$ in Definition \ref{def:orderbound} as the generalized order bound for the minimum distance of $C_L(D, G)^\perp$ or, simply, the \textit{order bound}.

\begin{lemma}
  \label{lem:4g_minus_1}
  Let $D$ and $G$ as in Definition \ref{def:orderbound}. Let $g := g(\mathcal{X})$. Then $d(G)\geq \mathrm{deg}(G)-2g+2$. 
  If $\mathrm{deg}(G) \geq 4g - 1$, the equality $d(G) = \mathrm{deg}(G) - 2g + 2$ holds.
\end{lemma}

\begin{proof}
  The inequality $d(G) \geq \mathrm{deg}(G)-2g+2$ follows directly from \cite{OrderBound}*{Proposition 10}. Assume $\mathrm{deg}(G) \geq 4g - 1$. 
  Fix a sequence $S = R_1, R_2, \dots$ of points of $\mathcal{X} \setminus \mathrm{supp}(D)$. As in particular $\mathrm{deg}(G) \geq 2g - 1$, the Riemann-Roch theorem implies that $L(G + R_1 + \cdots + R_i) \neq L(G + R_1 + \cdots + R_{i+1})$ for all $i \geq 0$. Moreover, it follows from \cite{OrderBound}*{Remark 5} that $\nu(R_{i+1}; 0, G + R_1 + \cdots + R_i) = \mathrm{deg}(G + R_1 + \cdots + R_{i}) - 2g + 2$ for all $i \geq 0$. As a consequence,
  $$ d_S(G) = \nu(R_1; 0, G) = \mathrm{deg}(G) - 2g + 2. $$
  Since $d_S(G)$ does not depend on the chosen sequence $S$, the conclusion follows.
\end{proof}
Lemma \ref{lem:4g_minus_1} shows that the order bound $d(G)$ coincides with the Goppa bound if $G$ has degree larger than or equal to $4g-1$. At the same time, the order bound cannot be worse than the Goppa bound for $\mathrm{deg}(G) < 4g-1$.

\begin{remark}
  \label{remark:restriction}
  Though the order bound $d(G)$ can be obtained theoretically by considering all possible sequences of points that do not occur in the support of $D$, this is not feasible in practice unless we restrict the set of possible sequences to a finite set. A first step into this direction is to observe that the computation of $d(G)$ using Definition \ref{def:orderbound} is only needed when $\mathrm{deg}(G) < 4g-1$ (see Lemma \ref{lem:4g_minus_1}) and that, in this case, only the first $4g-1-\mathrm{deg}(G)$ entries from every sequence $S$ are relevant to define $d(G)$. However, some additional condition must be imposed: for example, at the cost of obtaining a possibly worse bound, one restricts the choice of the points that can occur in a sequence $S$ to a finite set of points $\mathcal{P}$, chosen beforehand. For practical convenience, the set $\mathcal{P}$ can be chosen as the set of rational points of $\mathcal{X}$ that are not in the support of $D$.
\end{remark}

Throughout the rest of the paper we will apply the restriction suggested in Remark \ref{remark:restriction} when needed for practical purpose. With slight abuse of notation we will continue denoting the bound with $d(G)$ and we will refer to it as the \textit{order bound}. Note that this choice does not affect the statements of Proposition \ref{prop:orderbound} and Lemma \ref{lem:4g_minus_1}.

In Section \ref{section:computation} we will use Proposition \ref{prop:orderbound} to obtain a lower bound on the minimum distance of duals of two-point AG codes. A two-point AG code is an AG code $C_L(D, G)$ such that the support of the divisor $G$ consists of two distinct points $Q$ and $P$ only, namely $G = aQ + bP$ for some $a, b \in \mathbb{Z}$. 

\begin{remark}
  The dual of the two-point code $C_L(D, G)$ is not necessarily a two-point code; in fact, it is known that $C_L(D, G)^\perp = C_L(D, H)$ for some divisor $H$ whose support is disjoint from the support of $D$, but the support of $H$ might consist of more than two points. See \cite{Stichtenoth}*{Proposition 2.2.10} for more details.
\end{remark}

For a two-point code $C_L(D, G)$, we typically choose $Q$ and $P$ to be rational points of $\mathcal{X}$ and $D$ to be the sum of all rational points of $\mathcal{X}$ different from $Q$ and $P$. In such case, to the aim of computing the order bound $d(G)$ on the minimum distance of the code $C_L(D, G)^\perp$, one can consider sequences $S = R_1, R_2, \dots$ with $R_i \in \{ Q, P \}$ for all $i \geq 1$. In particular
\begin{alignat*}{2}
  \nu(Q; 0, aQ + bP) & = \# \{ (i, j) \in H(Q) \times H(Q; bP) \mid i+j = a+1 \}, \\
  \nu(P; 0, aQ + bP) & = \# \{ (i, j) \in H(P) \times H(P; aQ) \mid i+j = b+1 \}.
\end{alignat*}
We conclude this section with an exposition of a techinque, also used in \cite{TwoPointCodes}, for conveniently computing the sets $H(Q; bP)$ and $H(P; aQ)$. Denote with $\mathcal{R}(Q, P)$ the ring of functions in $\mathbb{F}_q(\mathcal{X})$ that are regular outside $Q$ and $P$, namely
$$ \mathcal{R}(Q, P) := \{ f \in \mathbb{F}_q(\mathcal{X}) \mid v_R(f) \geq 0 \; \forall R \neq Q, P \}. $$
The two-point Weierstrass semigroup $H(Q, P)$ can be defined as the set
$$ H(Q, P) := \{ (i, j) \in \mathbb{Z}^2 \mid \exists f \in \mathcal{R}(Q, P) \setminus \{0\}, v_Q(f) = -i, v_P(f) = -j \}. $$
This generalization of the classical Weierstrass semigroup at one point has been studied for example in \cite{GeneralizationWeierstrassSemigroup}. We define the \textit{period} of the two-point Weierstrass semigroup $H(Q, P)$ as
$$ \pi := \min \{ k \in \mathbb{N} \setminus \{ 0 \} \mid k(Q - P) \; \text{is a principal divisor} \} $$
and we define the map
\begin{alignat*}{2}
  \tau_{Q, P} : \; & \mathbb{Z} && \longrightarrow \mathbb{Z} \\
  & i && \longmapsto \min \{ j \mid (i, j) \in H(Q, P) \}.
\end{alignat*}
Some of the properties of the map $\tau_{Q, P}$ are summarized in the following proposition. See \cite{GeneralizationWeierstrassSemigroup}*{Proposition 14, Proposition 17} for details.

\begin{proposition}
  \label{prop:properties_tau}
  Let $\pi$ be the period of the two-point Weierstrass semigroup $H(Q, P)$ and $g = g(\mathcal{X})$ be the genus of $\mathcal{X}$. Then:
  \begin{itemize}
  \item [a)] $\tau_{Q, P}$ is bijective, with inverse map $\tau^{-1}_{Q, P} = \tau_{P, Q}$;
  \item [b)] $-i \leq \tau_{Q, P}(i) \leq 2g - i$ for all $i \in \mathbb{Z}$;
  \item [c)] $\tau_{Q, P}(i + \pi) = \tau_{Q, P}(i) - \pi$;
  \item [d)] $\sum_{i=c}^{\pi+c-1} (i + \tau_{Q, P}(i)) = \pi g$ for all $c \in \mathbb{Z}$.
  \end{itemize}
\end{proposition}

For computational purposes, it is convenient to provide a method to describe the map $\tau_{Q, P}^{-1} = \tau_{P, Q}$. With the following proposition we show that $\tau_{Q, P}^{-1}(j)$ can be computed efficiently for all $j \in \mathbb{Z}$.

\begin{proposition}
  \label{prop:tau_inverse}
  Let $\pi$ be the period of the two-point Weierstrass semigroup $H(Q, P)$. Let $j \in \mathbb{Z}$ and $i := \tau_{Q, P}^{-1}(j)$. Then $i = i^\prime - j + \tau_{Q, P}(i^\prime)$, where $i^\prime$ is the unique integer in $\{ 0, \dots, \pi-1 \}$ such that $\tau_{Q, P}(i^\prime) \equiv j \Mod{\pi}$.
\end{proposition}

\begin{proof}
  From Proposition \ref{prop:properties_tau}, $\tau_{Q, P}(a + \pi) = \tau_{Q, P}(a) - \pi$ for all $a \in \mathbb{Z}$ and $\tau_{Q, P}$ is bijective; thus $\{ \tau_{Q, P}(a) \mid 0 \leq a < \pi \}$ is a complete set of representatives of congruence classes modulo $\pi$. In particular, there exists a unique $i^\prime \in \{ 0, \dots, \pi-1 \}$ such that $\tau_{Q, P}(i^\prime) \equiv j \Mod{\pi}$. Write $\tau_{Q, P}(i) = j = \tau_{Q, P}(i^\prime) + (j - \tau_{Q, P}(i^\prime))$. Then
  \begin{equation}
    \label{eq:tau_inverse}
    \tau_{Q, P}(i^\prime) = \tau_{Q, P}(i) - (j - \tau_{Q, P}(i^\prime)) = \tau_{Q, P}(i + (j - \tau_{Q, P}(i^\prime))),
  \end{equation}
  where the last equality follows from Proposition \ref{prop:properties_tau} (c) as $j - \tau_{Q, P}(i^\prime)$ is a multiple of $\pi$. Applying $\tau_{Q, P}^{-1}$ to the left and the right side of equation \eqref{eq:tau_inverse}, we get $i^\prime = i + j - \tau_{Q, P}(i^\prime)$. The conclusion follows.
\end{proof}

We recall two useful results that rely on the knowledge of the function $\tau_{Q, P}$ only. The first one allows the determination of the dimension of the Riemann-Roch space of a two-point divisor $G=aQ + bP$, $a, b\in \mathbb{N}$. The second one provides an explicit expression for the set of $G$-non-gaps at $Q$ and the set of $G$-non-gaps at $P$. These results are crucial for the computation of the order bound.

\begin{theorem}[\cite{TwoPointCodes}*{Theorem 9}]
  \label{theorem:dim_LG}
  Let $G = aQ + bP$ with $a, b \in \mathbb{N}$. The Riemann-Roch space $L(G)$ has dimension $|\{i \leq a \mid \tau_{Q, P}(i) \leq b \}|$.
\end{theorem}

\begin{corollary}[\cite{TwoPointCodes}*{Corollary 10}]
  \label{cor:tau-non-gaps}
  Let $G = aQ + bP$ with $a, b \in \mathbb{N}$. Then $H(Q; G) = \{ i \in \mathbb{Z} \mid \tau_{Q, P}(i) \leq b \}$ and $H(P; G) = \{ i \in \mathbb{Z} \mid \tau_{Q, P}^{-1}(i) \leq a \}$.
\end{corollary}

Another important consequence of Theorem \ref{theorem:dim_LG} that we have not explicitly found in literature, but is worth to mention, is that a possible way of determining the full two-point Weierstrass semigroup $H(Q, P)$ from the sole knowledge of the map $\tau_{Q, P}$ is the following.

\begin{corollary}
  $H(Q, P) = \{ (i, j) \in \mathbb{Z}^2 \mid \tau_{Q, P}(i) \leq j, \tau_{Q, P}^{-1}(j) \leq i \}.$
\end{corollary}

\begin{proof}
  By definition, the pair $(i, j)$ belongs to $H(Q, P)$ if and only if there exists a function $f$ such that $(f)_\infty = iQ + jP$. This is equivalent to saying that $L(iQ + jP) \neq L((i-1)Q + jP)$ and $L(iQ + jP) \neq L(iQ + (j-1)P)$. From the inclusions $L((i-1)Q + jP) \subseteq L(iQ + jP)$ and $L(iQ + (j-1)P) \subseteq L(iQ + jP)$ and from the fact that Riemann-Roch spaces have a vector space structure, $(i, j)$ belongs to $H(Q, P)$ if and only if
  $$\begin{cases}
    \dim{L(iQ + jP)} \neq \dim{L((i-1)Q + jP)}, \\
    \dim{L(iQ + jP)} \neq \dim{L(iQ + (j-1)P)},
  \end{cases}$$
  which can be reformulated using Theorem \ref{theorem:dim_LG} as
  $$\begin{cases}
    |\{ i^\prime \leq i \mid \tau_{Q, P}(i^\prime) \leq j \}| \neq |\{ i^\prime \leq i-1 \mid \tau_{Q, P}(i^\prime) \leq j \}|, \\
    |\{ j^\prime \leq j \mid \tau_{Q, P}^{-1}(j^\prime) \leq i \}| \neq |\{ j^\prime \leq j-1 \mid \tau_{Q, P}^{-1}(j^\prime) \leq i \}|.
  \end{cases}$$
   Note that we can use Theorem \ref{theorem:dim_LG} when at least one of the coefficients of the two-point divisor is negative, as this is not required in the proof of Theorem \ref{theorem:dim_LG} given in \cite{TwoPointCodes}. The two conditions of the above system are equivalent to the conditions $\tau_{Q, P}(i) \leq j$ and $\tau_{Q, P}^{-1}(j) \leq i$ respectively.
\end{proof}

\begin{remark}
  The fact that $H(Q, P)$ can be fully determined from the sole knowledge of $\tau_{Q, P}$ can also be deduced from \cite{TwoPointsSuzuki}*{Lemma 3.2}. Although our definition of two-point Weierstrass semigroup is slightly different from the one adopted in \cite{TwoPointsSuzuki}, we can adapt the methods used in \cite{TwoPointsSuzuki} to our setting: we define an operation of addition between two pairs of integers $(i_1, j_1)$ and $(i_2, j_2)$ as follows:
  $$ (i_1, j_1) + (i_2, j_2) := (\max \{ i_1, i_2 \}, \max \{ j_1, j_2 \}). $$
  With this operation, the two-point Weierstrass semigroup $H(Q, P)$ is generated by the pairs $(i, \tau_{Q, P}(i))$ for $i \in \mathbb{Z}$. Moreover, all pairs $(i, \tau_{Q, P}(i))$ with $i \in \mathbb{Z}$ can be obtained from a minimal set $\{ (i, \tau_{Q, P}(i)) \mid 0 \leq i \leq \pi - 1 \}$ using Proposition \ref{prop:properties_tau} (c). In fact, for a fixed $i \in \mathbb{Z}$, there exist unique $i^\prime \in \{ 0, \dots, \pi - 1 \}$ and $k \in \mathbb{Z}$ such that $(i, \tau_{Q, P}(i)) = (i^\prime + k \pi, \tau_{Q, P}(i^\prime) - k \pi)$.
\end{remark}

\section{The Weierstrass semigroup at a certain pair of points on the $\mathcal{BM}_n$ curve}
\label{section:two_point_semigroup}

Let $q$ be a prime power, $n \geq 3$ be an odd integer and $m := (q^n+1)/(q+1)$. We devote this section to the study of a particular two-point Weierstrass semigroup, which will be specified later on, of the curve $\mathcal{BM}_n$ defined in \eqref{eq:BM_definition}. Firstly, we summarize some of the main properties of the curve $\mathcal{BM}_n$. Further details can be found in \cite{ANewFamily}.

\begin{proposition}
  Let $\mathcal{BM}_n$ be the Beelen-Montanucci curve defined in \eqref{eq:BM_definition} and $\mathcal{H}_q$ the Hermitian curve defined in \eqref{eq:Hermitian}.
  \begin{itemize}
  \item $\mathcal{BM}_n$ is maximal over the field $\mathbb{F}_{q^{2n}}$.
  \item $\mathcal{BM}_n$ has genus $g(\mathcal{BM}_n) = \frac{1}{2} (q-1)(q^{n+1}+q^n-q^2)$ and $N_n = q^{2n+2} - q^{n+3} + q^{n+2} + 1$ $\mathbb{F}_{q^{2n}}$-rational points.
  \item The $q^3+1$ $\mathbb{F}_{q^2}$-rational points of $\mathcal{H}_q$ are totally ramified in the cover $\mathcal{BM}_n \to \mathcal{H}_q$.
  \item The full automorphism group $\mathrm{Aut}(\mathcal{BM}_n)$ of $\mathcal{BM}_n$ is isomorphic to $\mathrm{SL}(2, q) \rtimes C_{q^n+1}$, where $C_{q^n+1}$ is the cyclic group with $q^n+1$ elements, and acts on the set of $\mathbb{F}_{q^2}$-rational points of $\mathcal{BM}_n$ with two orbits
    $$ O_1 := \{ P_1, \dots, P_{q+1} \} \quad \text{and} \quad O_2 := \{ Q_1, \dots, Q_{q^3-q} \}, $$
    where $P_1, \dots, P_{q+1}$ are the $q+1$ $\mathbb{F}_{q^2}$-rational points of $\mathcal{BM}_n$ lying over the $q+1$ points at infinity of $\mathcal{H}_q$ and $Q_1, \dots, Q_{q^3-q}$ are the $q^3-q$ $\mathbb{F}_{q^2}$-rational points of $\mathcal{BM}_n$ lying over the remaining $\mathbb{F}_{q^2}$-rational points of $\mathcal{H}_q$.
  \end{itemize}
\end{proposition}

Note that the points in $O_1$ can be parametrized in homogeneous coordinates $(x : y : z : w)$ by $P_i = (1 : a_i : 0 : 0)$ with $a_i^{q+1} = 1$. Let $M := (m-1)/(q^2-q)$. From \cite{AGcodes}*{Theorem 1.1}, the Weierstrass semigroup $H(P)$ at any point $P \in O_1$ is
\begin{equation}
  \label{eq:semigroup_O1}
  H(P) = \langle q^n+1, mq+k(q^2-q) \mid k = 0, \dots, M \rangle
\end{equation}
and the Weierstrass semigroup $H(Q)$ at any point $Q \in O_2$ is
\begin{equation}
  \label{eq:semigroup_O2}
  H(Q) = \langle q^n+1-m, q^n+1-k \mid k = 0, \dots, M \rangle.
\end{equation}

Since the Weierstrass semigroup at any point is invariant under the action of the automorphism group $\mathrm{Aut}(\mathcal{BM}_n)$ on that point (see for example \cite{Stichtenoth}*{Lemma 3.5.2}), points in the same orbit have the same Weierstrass semigroup. In the following, we will choose
$$ P_1 := (1 : -1 : 0 : 0) \quad \text{and} \quad Q_1 := (1 : 0 : 0 : 1) $$
as representatives of the orbits $O_1$ and $O_2$ respectively.

We recall some functions in $\mathbb{F}_{q^{2n}}(\mathcal{BM}_n)$ and their principal divisors, that will be useful in the following. Let
$$ \alpha := \frac{x-1}{x+y} \quad \text{and} \quad \theta_k := \frac{z^k}{x+y} $$
for $k = 0, \dots, M$. From \cite{AGcodes}*{Lemma 3.1, Lemma 3.3}:
\begin{align}
  (x + y) & = mq P_1 - m \sum_{i=2}^{q+1} P_i, \label{eq:divisor_xplusy} \\
  (\alpha) & = (q^n+1)(Q_1 - P_1), \label{eq:divisor_alpha} \\
  (\theta_k) & = k \sum_{i = 1}^{q^3-q} Q_i + (m - k(q^2-q)) \sum_{i=2}^{q+1} P_i - (mq+k(q^2-q)) P_1. \label{eq:divisor_thetak}
\end{align}

\begin{lemma}
  \label{lemma:divisor_thetatilde}
  Let $\tilde{\theta}_0 := \theta_0 - 1$. The principal divisor of $\tilde{\theta}_0$ in $\mathbb{F}_{q^{2n}}(\mathcal{BM}_n)$ is $(\tilde{\theta}_0) = m Q_1 + E - mq P_1$, where $E$ is an effective divisor whose support does not contain $Q_1$ and $P_1$.
\end{lemma}

\begin{proof}
  Define $t := x + y - 1$, so that $\tilde{\theta}_0 = -t/(x+y)$ and $\tilde{\theta}_0$ has principal divisor $(\tilde{\theta}_0) = (t) - (x+y)$. Let $\overline{P}_1, \dots, \overline{P}_{q+1}$ be the $q+1$ points at infinity of the Hermitian curve $\mathcal{H}_q$, namely $\overline{P}_i := (1 : a_i : 0)$ in homogeneous coordinates $(x : y : w)$, with $a_i^{q+1} = 1$. The line defined by $t$ intersects $\mathcal{H}_q$ in exactly $q+1$ distinct points; these points are $\overline{Q}_b := (1-b : b : 1)$ for $b \in \mathbb{F}_{q^2}$ satisfying $b^q + b = 0$ and $\overline{P}_1$.  Then, the principal divisor of $t$ in $\mathbb{F}_{q^{2n}}(\mathcal{H}_q)$ is:
  $$ (t)_{\mathbb{F}_{q^{2n}}(\mathcal{H}_q)} = \overline{P}_1 + \sum_{\substack{b \in \mathbb{F}_{q^2} \\ b^q + b = 0}} \overline{Q}_b - \sum_{i=1}^{q+1} \overline{P}_i = \sum_{\substack{b \in \mathbb{F}_{q^2} \\ b^q + b = 0}} \overline{Q}_b - \sum_{i=2}^{q+1} \overline{P}_i. $$
  For all $b \in \mathbb{F}_{q^2}$ satisfying $b^q + b = 0$ let $Q_{i_b} := (1-b : b : 0 : 1)$ be the unique point in $O_2$ lying over $\overline{Q}_b$ and note that $Q_1 = Q_{i_0}$. Further, $P_i$ is the unique point in $O_1$ lying over $\overline{P}_i$ for all $i = 1, \dots, q+1$. Then, the principal divisor of $t$ in $\mathbb{F}_{q^{2n}}(\mathcal{BM}_n)$ is
  $$ (t) = m \sum_{\substack{b \in \mathbb{F}_{q^2} \\ b^q + b = 0}} Q_{i_b} - m \sum_{i=2}^{q+1} P_i. $$
  Hence,
  \begin{align*}
    (\tilde{\theta}_0) & = m \sum_{\substack{b \in \mathbb{F}_{q^2} \\ b^q + b = 0}} Q_{i_b} - m \sum_{i=2}^{q+1} P_i - mq P_1 + m \sum_{i=2}^{q+1} P_i \\
                       & = m Q_1 + m \sum_{\substack{b \in \mathbb{F}_{q^2}, b \neq 0 \\ b^q + b = 0}} Q_{i_b} - mq P_1.
  \end{align*}
\end{proof}

We are now ready to focus on the two-point Weierstrass semigroup $H(Q_1, P_1)$. To this aim, we first give an explicit description of the ring of functions that are regular outside $Q_1$ and $P_1$ and we compute the period of $H(Q_1, P_1)$. 

\begin{proposition}
  \label{prop:ring-reg-func}
  $\mathcal{R}(Q_1, P_1) = \mathbb{F}_{q^{2n}} [\alpha, \alpha^{-1}, \tilde{\theta}_0, \theta_1, \theta_2, \dots, \theta_M]$.
\end{proposition}

\begin{proof}
  From \eqref{eq:divisor_alpha}, \eqref{eq:divisor_thetak} and Lemma \ref{lemma:divisor_thetatilde} it is clear that the $\mathbb{F}_{q^{2n}}$-rational functions $\alpha, \tilde{\theta_0}, \theta_1, \dots, \theta_M$ are regular outside $P_1$; from \eqref{eq:semigroup_O1} it follows that
  $$ H(P_1) = \langle -v_{P_1}(\alpha), -v_{P_1}(\tilde{\theta_0}), -v_{P_1}(\theta_1), \dots, -v_{P_1}(\theta_M) \rangle. $$
  We first prove that the ring $\mathcal{R}(P_1) := \bigcup_{i\geq 0} L(iP_1)$ of $\mathbb{F}_{q^{2n}}$-rational functions that are regular outside $P_1$ is
  \begin{equation}
    \label{eq:regular_functions_P1}
    \mathcal{R}(P_1) = \mathbb{F}_{q^{2n}}[\alpha, \tilde{\theta_0}, \theta_1, \dots, \theta_M].  
  \end{equation}
  It is clear that $\mathbb{F}_{q^{2n}}[\alpha, \tilde{\theta_0}, \theta_1, \dots, \theta_M] \subseteq \bigcup_{i\geq 0} L(iP_1)$. In fact, for each function $h$ of the ring $\mathbb{F}_{q^{2n}}[\alpha, \tilde{\theta_0}, \theta_1, \dots, \theta_M]$, being $h$ a combination of $\alpha, \tilde{\theta_0}, \theta_1, \dots, \theta_M$, there exists a positive integer $\hat{\gamma}$ such that $h \in L(\hat{\gamma} P_1)$. \\
  Conversely, if $h \in \bigcup_{i\geq 0} L(iP_1)$, then in particular $h \in L(i_h P_1)$ for some $i_h \geq 0$. We prove that $h$ belongs to $\mathbb{F}_{q^{2n}}[\alpha, \tilde{\theta_0}, \theta_1, \dots, \theta_M]$ by induction on $i_h$. If $i_h = 0$ then trivially $h$ belongs to $L(0)=\mathbb{F}_{q^{2n}}\subseteq \mathbb{F}_{q^{2n}}[\alpha, \tilde{\theta_0}, \theta_1, \dots, \theta_M]$. We proceed now to the induction step. Assume that the claim holds for all integers $i_h$ less than or equal to $\tilde{n}$ and consider $i_h = \tilde{n}+1$. If $\tilde{n}+1$ is not an element of $H(P_1)$, then $h \in L(kP_1)$ for some $k \leq \tilde{n}$, and the thesis follows by induction. If instead $\tilde{n}+1$ belongs to $H(P_1)$, then $\tilde{n}+1$ can be written as a combination of $-v_{P_1}(\alpha), -v_{P_1}(\tilde{\theta_0}), -v_{P_1}(\theta_1), \dots, -v_{P_1}(\theta_M)$, namely
  $$ \tilde{n} + 1 = a_1 (-v_{P_1}(\alpha)) +\cdots + a_{M+2}(-v_{P_1}(\theta_M)) $$  
  for some $a_i\in \mathbb{N}$, $1\leq i\leq \ M+2$. Then note that the pole-divisor $(h)_\infty$ of $h$ is
  $$ (h)_\infty = (\alpha^{a_1} \cdot \tilde{\theta_0}^{a_2} \cdot \theta_1^{a_3} \cdots \theta_M^{a_{M+2}})_\infty $$
  and hence there exists $\lambda\in\mathbb{F}_{q^{2n}}\setminus\{0\}$ such that $h^\prime := h - \lambda \alpha^{a_1} \cdot \tilde{\theta_0}^{a_2} \cdots \theta_1^{a_3} \cdots \theta_M^{a_{M+2}}$ is an element of $\bigcup_{i\geq 0}L(iP_1)$ with $v_{P_1}(h^\prime) > - (\tilde{n}+1)$. By the induction hypothesis $h^\prime \in \mathbb{F}_{q^{2n}}[\alpha, \tilde{\theta_0}, \theta_1, \dots, \theta_M]$ and so 
  $$ h = h^\prime + \lambda \alpha^{a_1} \cdot \tilde{\theta_0}^{a_2} \cdots \theta_1^{a_3} \cdots \theta_M^{a_{M+2}}\in \mathbb{F}_{q^{2n}}[\alpha, \tilde{\theta_0}, \theta_1, \dots, \theta_M]. $$
  The statement of the proposition now follows: it is clear from \eqref{eq:divisor_alpha} and \eqref{eq:regular_functions_P1} that any function in $\mathbb{F}_{q^{2n}}[\alpha, \alpha^{-1}, \tilde{\theta_0}, \theta_1, \dots, \theta_M]$ is regular outside $Q_1$ and $P_1$; conversely, for any $f \in \mathcal{R}(Q_1, P_1)$ there exists a suitable integer $k \geq 0$ such that $f \alpha^k$ belongs to $\mathcal{R}(P_1)$. This shows that $f$ belongs to $\mathbb{F}_{q^{2n}}[\alpha, \alpha^{-1}, \tilde{\theta_0}, \theta_1, \dots, \theta_M]$.
\end{proof}

\begin{lemma}
  \label{lem:period-two-point}
  The period $\pi$ of the Weierstrass semigroup $H(Q_1, P_1)$ is $\pi = q^n + 1$.
\end{lemma}

\begin{proof}
  Assume by contradiction that $k(Q_1 - P_1)$ is a principal divisor for some $k \in \{ 1, \dots, q^n \}$. Let $f \in \mathbb{F}_{q^{2n}}(\mathcal{BM}_n)$ such that $(f) = k(Q_1 - P_1)$. In particular $k$ is a non-gap of the Weierstrass semigroup $H(Q_1)$, as $v_{Q_1}(f^{-1}) = -k$ and $Q_1$ is the only pole of $f^{-1}$. The smallest non-zero element of $H(Q_1)$ is $q^n+1-m$ (see \eqref{eq:semigroup_O2}), hence $q^n+1-m \leq k \leq q^n$ and we can write $k = q^n+1-m+j$ for some $j \in \{ 0, \dots, m-1 \}$. Since
  $$ (\alpha^{-1} f) = (q^n+1-k)(P_1 - Q_1) = (m-j)(P_1 - Q_1), $$
  then $m-j$ must be a non-gap of the Weierstrass semigroup $H(Q_1)$; this is not possible, as $0 < m-j < qm = q^n+1-m$.
\end{proof}

Given Proposition \ref{prop:ring-reg-func} and Lemma \ref{lem:period-two-point}, we are now able to prove the main theorem of the paper, which provides the explicit expression of the function $\tau_{Q_1, P_1}$. Note that the knowledge of the function $\tau_{Q_1, P_1}$ is sufficient to determine the two-point Weierstrass semigroup $H(Q_1, P_1)$.

\begin{theorem}
  \label{theorem:tau_BM}
  Let $i \in \mathbb{Z}$ and write $i = -k(q^n+1) - \ell m - \beta$ for a unique triple $(k, \ell, \beta) \in \mathbb{Z}^3$ such that $0 \leq \beta < m$, $0 \leq \ell < q+1$. Let $\gamma := \lceil \beta / M \rceil$. Then
  $$ \tau_{Q_1, P_1} (i) = k(q^n+1) + (\gamma + \ell)mq + \beta (q^2-q). $$
\end{theorem}

\begin{proof}
  Define the map $\tilde{\tau} : \mathbb{Z} \to \mathbb{Z}$ such that $\tilde{\tau}(i) = k(q^n+1) + (\gamma + \ell)mq + \beta (q^2-q)$ for all $i \in \mathbb{Z}$ and $k, \ell, \beta, \gamma$ as in the assumptions. We will prove that $\tilde{\tau}(i) = \tau_{Q, P}(i)$ for all $i \in \mathbb{Z}$. In the following, we fix $i \in \mathbb{Z}$, so that $k, \ell, \beta, \gamma$ are fixed too. Choose $M$ non-negative integers $i_1, \dots, i_M$ such that
  $$ \sum_{j=1}^M i_j j = \beta \quad \text{and} \quad \sum_{j=1}^M i_j = \gamma. $$
  Such choice of $i_1, \dots, i_M$ always exists: letting $\beta^\prime := \beta \mod M$, if $\beta^\prime \neq 0$ one can choose $i_M = \gamma-1$, $i_{\beta^\prime} = 1$ and $i_j = 0$ for $j \neq M, \beta^\prime$. If $\beta^\prime = 0$, one can instead choose $i_M = \gamma$ and $i_j = 0$ for $j \neq M$.\\
  Consider now the function
  $$ f := \alpha^k \tilde{\theta_0}^\ell \prod_{j=1}^M \theta_j^{i_j}. $$
  From \eqref{eq:divisor_alpha}, \eqref{eq:divisor_thetak} and Lemma \ref{lemma:divisor_thetatilde}, the principal divisor of $f$ is
  \begin{alignat*}{2}
    (f) & = && k(q^n+1)(Q_1 - P_1) + \ell ( m Q_1 + E - mq P_1) + \\
    &   && \sum_{j=1}^M i_j \left( j \sum_{i = 1}^{q^3-q} Q_i + (m - j(q^2-q)) \sum_{i=2}^{q+1} P_i - (mq + j(q^2-q)) P_1 \right) \\
    & = && \left( k(q^n+1) + \ell m + \sum_{j=1}^M i_j j \right) Q_1 + E^\prime - \\
    &   && \left( k(q^n+1) + \left( \sum_{j=1}^M i_j + \ell \right) mq + \sum_{j=1}^M i_j j (q^2-q) \right) P_1 \\
    & = && -i Q_1 + E^\prime - \tilde{\tau}(i) P_1.
  \end{alignat*}
  where $E$ and $E^\prime$ are effective divisors whose supports do not contain $Q_1$ and $P_1$. The above computation shows that $(i, \tilde{\tau}(i))$ belongs to $H(Q_1, P_1)$ and thus $\tilde{\tau}(i) \geq \tau_{Q_1, P_1}(i)$ by definition of $\tau_{Q_1, P_1}$.

  Finally, we can use Lemma \ref{prop:properties_tau} \textit{d)} to show that the equality $\tilde{\tau}(i) = \tau_{Q_1, P_1}(i)$ holds. In fact, we have just proved that $\tilde{\tau}(i) \geq \tau_{Q_1, P_1}(i)$ for all $i \in \mathbb{Z}$ and therefore
  \begin{equation}
    \label{eq:sum_equal_pig}
    \sum_{i=c}^{\pi+c-1} (i + \tilde{\tau}(i)) \geq \sum_{i=c}^{\pi+c-1} (i + \tau_{Q, P}(i)) = \pi g
  \end{equation}
  for all $c \in \mathbb{Z}$. To conclude, it is enough to check that the left side of equation \eqref{eq:sum_equal_pig} is equal to $\pi g$. We can choose $c = -\pi + 1$ without loss of generality, so that
  \begin{equation}
    \label{eq:sum_i_tilde_i}
    \sum_{i=-\pi+1}^0 (i + \tilde{\tau}(i)) = \sum_{\beta=0}^{m-1} \sum_{\ell=0}^q (-m \ell - \beta + (\gamma + \ell)mq + \beta(q^2-q)).
  \end{equation}
  Writing $\gamma = \frac{1}{M}(\beta + (M-\beta) \, \mathrm{mod} \, M)$, the quantity on the right side of equation \eqref{eq:sum_i_tilde_i} yields
  \begin{align*}
    & - \frac{m^2 q(q+1)}{2} - \frac{m(m-1)(q+1)}{2} + \frac{m^2(m-1)q(q+1)}{2M} + \\
    & \frac{mq^2 (q^2-1)(M-1)}{2} + \frac{m^2 q^2 (q+1)}{2} + \frac{m(m-1)q(q^2-1)}{2}.
  \end{align*}
  It can be checked with a direct computation that the above quantity is equal to $\frac{1}{2}(q^n+1)(q-1)(q^{n+1}+q^n-q^2) = \pi g$.
\end{proof}

\begin{figure}[ht!]
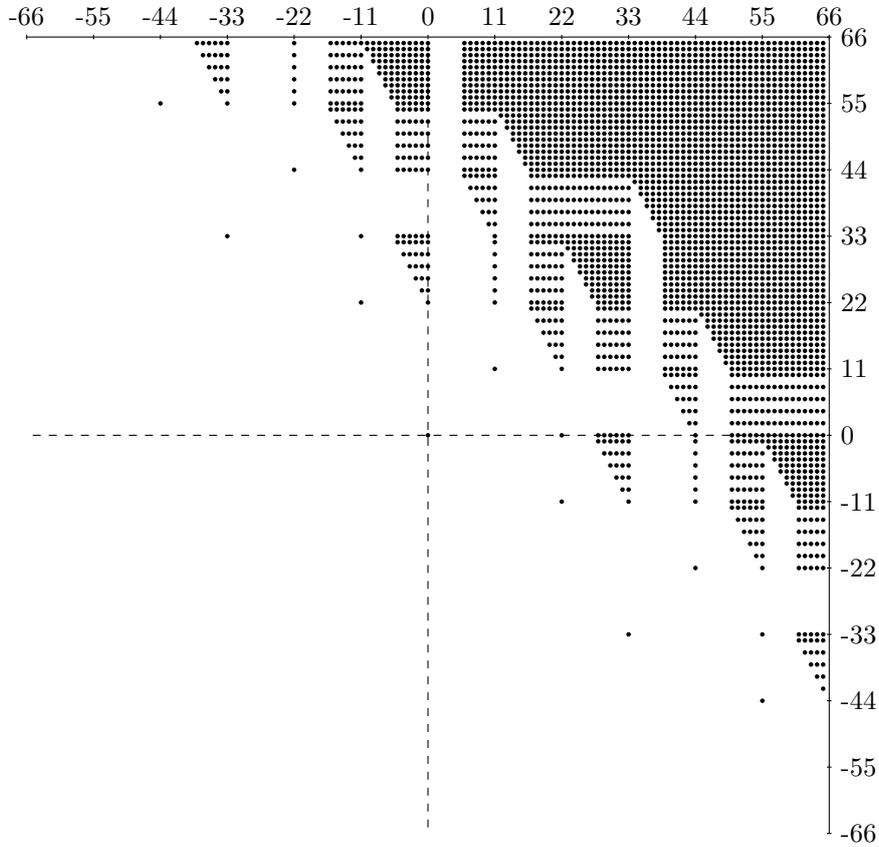

  \centering

  \caption{The two-point Weierstrass semigroup $H(Q_1, P_1)$ of $\mathcal{BM}_n$ for $q=2$ and $n=5$, of period $\pi=33$. Only the pairs $(i, j) \in H(Q_1, P_1)$ with $-2 \pi < i, j < 2 \pi$ are represented.}
\end{figure}

\section{Computation of the order bound and results}
\label{section:computation}

We are now ready to compute the order bound for dual codes of two-point AG codes from the $\mathcal{BM}_n$ curve, for all $n \geq 3$ odd. Let $G = aQ_1 + bP_1$ and denote with $\delta := a+b$ its degree. We define the divisor $D$ to be the sum of all the $\mathbb{F}_{q^{2n}}$-rational points of $\mathcal{BM}_n$ different from $Q_1$ and $P_1$. The degree of $D$ is therefore $\mathrm{deg}(D) = N_n - 2$, where $N_n = q^{2n+2} - q^{n+3} + q^{n+2} + 1$ is the number of $\mathbb{F}_{q^{2n}}$-rational points of $\mathcal{BM}_n$. The two-point AG code $C_L(D, G)$ and its dual $C_L(D, G)^\perp $ are linear subspaces of $\mathbb{F}_{q^{2n}}^{N_n-2}$.

If $\delta \geq N_n + 2g - 3$, then the code $C_L(D, G)^\perp$ is the zero code; this follows from the fact that, if $\delta \geq N_n + 2g - 3$, the divisors $G$ and $G-D$ are non-special, as their degrees exceed $2g - 2$ and, from the Riemann-Roch theorem, $\mathrm{dim}(C_L(D, G)) = \mathrm{dim}(L(G)) - \mathrm{dim}(L(G-D)) = N_n - 2$.

Define $\Delta := 4g - 1$. As pointed out in Section \ref{section:preliminary_results}, it is sufficient to determine the order bound for the code $C_L(D, G)^\perp$ for the case $\delta < \Delta$ only, since the order bound coincides with the Goppa bound if the degree of $G$ is larger than or equal to $\Delta$ (see Lemma \ref{lem:4g_minus_1}). The condition $\delta < \Delta$, which also implies $\mathrm{deg}(G) < \mathrm{deg}(D)$, makes the determination of the dimension of $C_L(D, G)^\perp$ a particularly easy task; in fact
$$ \mathrm{dim}(C_L(D, G)^\perp) = N_n - 2 - \mathrm{dim}(L(G)). $$
The dimension of $L(G)$ can be conveniently computed applying Theorem \ref{theorem:dim_LG} with the map $\tau_{Q_1, P_1}$ defined in Theorem \ref{theorem:tau_BM}.

The algorithm we propose for computing the order bound for $C_L(D, G)^\perp$ is inspired by \cite{TwoPointCodes}*{Algorithm 1} and takes into account the observations above. Similarly to \cite{TwoPointCodes}, we recursively obtain a bound for the minimum distance of the code $C_L(D, G)^\perp$ by successive iterations on the degree $\delta$ of $G$, starting from $\delta = \Delta-1$ and decreasing $\delta$ by $1$ at each round of the procedure until $\delta = 0$. Observe that it is easy to check if
\begin{equation}
  \label{eq:check_dimensions}
  \mathrm{dim}(L(a Q_1 + b P_1)) \neq \mathrm{dim}(L((a+1) Q_1 + b P_1)),
\end{equation}
since from Theorem \ref{theorem:dim_LG}, \eqref{eq:check_dimensions} holds if and only if $\tau_{Q_1, P_1}(a+1) \leq b$. Similarly, the inequality $\mathrm{dim}(L(a Q_1 + b P_1)) \neq \mathrm{dim}(L(a Q_1 + (b+1) P_1))$ holds if and only if $\tau_{Q_1, P_1}^{-1}(b+1) \leq a$. Note that $\tau_{Q_1, P_1}^{-1}(b+1)$ can be computed using Proposition \ref{prop:tau_inverse}.

\begin{algorithm}
  \label{algorithm:main}
  Input: a prime power $q$ and an odd integer $n \geq 3$. \\
  Output: a table $T$ whose rows consist of three cells: the first cell contains an integer $k$ representing the dimension of a code $C_L(D, a Q_1 + b P_1)^\perp$; the second cell contains a pair of integers $(a, b)$ such that $d(a Q_1 + b P_1) \geq d(a^\prime Q_1 + b^\prime P_1)$ for all codes $C_L(D, a^\prime Q_1 + b^\prime P_1)^\perp$ of dimension $k$; the third cell contains $d(a Q_1 + b P_1)$.
  \begin{enumerate}
  \item Initialize an empty table $T$.
  \item Define $g := \frac{1}{2} (q-1)(q^{n+1}+q^n-q^2)$ and $\Delta := 4g-1$.
  \item Construct an upper-left triangular matrix $A$ of size $(\Delta+1) \times (\Delta+1)$ and set $A[a, \Delta-a] = \Delta-2g+2$ for $a = 0, \dots, \Delta$.
  \item Define $\delta := \Delta-1$.
  \item For $a = 0, \dots, \delta$, define $b := \delta - a$ and
    \begin{align*}
      d_{Q_1} & := \begin{cases}
        \min \{ \nu(Q_1; 0, a Q_1 + b P_1), A[a+1, b] \} & \text{if} \; \tau_{Q_1, P_1}(a+1) \leq b, \\
        A[a+1, b] & \text{otherwise}, \\
      \end{cases} \\
      d_{P_1} & := \begin{cases}
        \min \{ \nu(P_1; 0, a Q_1 + b P_1), A[a, b+1] \} & \text{if} \; \tau_{Q_1, P_1}^{-1}(b+1) \leq a, \\
        A[a, b+1] & \text{otherwise}, \\
      \end{cases} \\
      d & := \max \{ d_{Q_1}, d_{P_1} \}.
    \end{align*}
  \item Compute $k := \mathrm{dim}(C_L(D, aQ_1 + bP_1)^\perp)$.
  \item Check if a row with value $k$ in the first cell exists in the table $T$.
    \begin{enumerate}
    \item If such row does not exist, add a new row to $T$ with $k$ in the first cell, $(a, b)$ in the second cell, $d$ in the third cell.
    \item If such row exists and $d$ is strictly larger than the value in the third cell, update the row by overwriting the pair in the second cell with $(a, b)$ and the value in the third cell with $d$.
    \item If such row exists and $d$ is not larger than the value in the third cell, do nothing.
    \end{enumerate}
  \item Redefine $\delta := \delta - 1$ and repeat the procedure from step 5 until $\delta = 0$.
  \end{enumerate}
\end{algorithm}

The table $T$ in output of Algorithm \ref{algorithm:main} stores the information on possible improvements on the minimum distance of codes $C_L(D, G)^\perp$ over the Goppa bound. Note for example that the improvements obtained for $q=2, n=3$ and $q=3, n=3$ are identical to the ones obtained in \cite{TwoPointCodes} (the case $q=2, n=3$ is summarized in Table \ref{table:q2n3}, compare with \cite{TwoPointCodes}*{Table 1}); on one hand this should not surprise, since $\mathcal{BM}_3$ and $GGS_3$ are isomorphic, but on the other hand it is interesting to see that our definition of order bound, which is slightly weaker than the one given in \cite{TwoPointCodes}, does not affect the estimate for the minimum distance in these particular cases.

\begin{table}[h!]
  \centering
  \begin{tabular}{|c|c|c||c|c|c||c|c|c|}
    \hline
    $k$ & $(a, b)$ & $d$ & $k$ & $(a, b)$ & $d$ & $k$ & $(a, b)$ & $d$ \\
    \hline
    195 & (0, 37) & 20 & 205 & (1, 26) & 11 & 215 & (1, 16) & 4 \\
    196 & (1, 35) & 19 & 206 & (1, 25) & 10 & 216 & (7, 8) & 4 \\
    197 & (1, 34) & 18 & 207 & (1, 24) & 9 & 217 & (1, 14) & 3 \\
    198 & (1, 33) & 17 & 208 & (1, 23) & 9 & 218 & (1, 13) & 3 \\
    199 & (1, 32) & 16 & 209 & (1, 22) & 8 & 219 & (1, 11) & 3 \\
    200 & (1, 31) & 15 & 210 & (0, 22) & 6 & 220 & (4, 7) & 2 \\
    201 & (0, 31) & 14 & 211 & (0, 21) & 6 & 221 & (2, 7) & 2 \\
    202 & (1, 29) & 13 & 212 & (0, 20) & 6 & 222 & (2, 5) & 2 \\
    203 & (4, 25) & 13 & 213 & (0, 19) & 6 & & & \\
    204 & (0, 28) & 12 & 214 & (1, 17) & 5 & & & \\
    \hline
  \end{tabular}
  \caption{Table $T$ obtained from Algorithm \ref{algorithm:main} with $q=2$, $n=3$ and code length $N_3-2 = 223$.}
  \label{table:q2n3}
\end{table}

We also compared our results obtained using Algorithm \ref{algorithm:main} with the results obtained using \cite{TwoPointCodes}*{Algorithm 1} for $q=2$ and $n=5$. In this case the two curves $\mathcal{BM}_5$ and $GGS_5$ are not isomorphic. Table \ref{table:q2n5} summarizes the cases where our results improve those from \cite{TwoPointCodes}.

\begin{table}[h!]
  \centering
  \begin{tabular}{|c|c|c|c||c|c|c|c|}
    \hline
    $k$ & $(a, b)$ & $d$ & $d_2$ & $k$ & $(a, b)$ & $d$ & $d_2$\\
    \hline
    3875 & (5, 132) & 52 & 51 & 3920 & (5, 87) & 17 & 16 \\
    3876 & (5, 131) & 51 & 50 & 3926 & (15, 71) & 14 & 12 \\
    3878 & (5, 129) & 49 & 48 & 3927 & (15, 70) & 14 & 12 \\
    3880 & (5, 127) & 47 & 46 & 3928 & (15, 69) & 13 & 11 \\
    3904 & (0, 108) & 28 & 27 & 3929 & (15, 68) & 13 & 11 \\
    3909 & (5, 98) & 23 & 22 & 3930 & (14, 68) & 12 & 11 \\
    3917 & (5, 90) & 19 & 18 & 3934 & (1, 77) & 8 & 7 \\
    \hline
  \end{tabular}
  \caption{For $q=2$, $n=5$, Table \ref{table:q2n5} reports the largest estimate for the minimum distance $d$ of a code $C_L(D, aQ_1 + bP_1)^\perp$ of length $N_5-2 = 3967$ and dimension $k$ and compares $d$ with $d_2$, the largest estimate obtained in \cite{TwoPointCodes} for codes of same length and dimension. Only the cases where $d > d_2$ are reported.}
  \label{table:q2n5}
\end{table}

\end{document}